\documentclass[11pt]{amsart}
\usepackage{amssymb}
\usepackage{times}
\usepackage{fullpage}

\numberwithin{equation}{section}
\newtheorem{theorem}{Theorem}[section]
\newtheorem{proposition}[theorem]{Proposition}
\newtheorem{corollary}[theorem]{Corollary}
\newtheorem{lemma}[theorem]{Lemma}
\theoremstyle{definition}
\newtheorem{definition}[theorem]{Definition}
\newtheorem{example}[theorem]{Example}
\DeclareMathOperator{\Sym}{Sym}
\newcommand{\Linf}{$L_\infty$}
\newcommand{\Nbar}{\bar{N}}

\usepackage{enumerate}

\usepackage[cmtip,curve,all]{xy}

\let\To=\to
\input diagxy
\renewcommand{\to}{\To}
\newcommand{\too}{\To}

\usepackage[bbgreekl]{mathbbol}
\usepackage[OT2,T1]{fontenc}
\DeclareSymbolFont{cyrletters}{OT2}{wncyr}{m}{n}
\DeclareMathSymbol{\Sh}{\mathalpha}{cyrletters}{"58}
\DeclareMathSymbol{\sh}{\mathalpha}{cyrletters}{"78}

\usepackage[initials,lite]{amsrefs}

\let\im=\undef
\DeclareMathOperator{\im}{im}
\DeclareMathOperator{\Map}{Hom}
\newcommand{\p}{\partial}
\newcommand{\bull}{\bullet}

\newcommand{\C}{\mathbb{C}}

\newcommand{\Z}{\mathbb{Z}}

\newcommand{\CX}{\mathcal{X}}

\newcommand{\eps}{\varepsilon}
\DeclareMathOperator{\MC}{\mathsf{MC}}
\DeclareMathOperator{\ad}{ad}

\newcommand{\CO}{\mathcal{O}}

\DeclareMathOperator{\End}{End}
\newcommand{\half}{\tfrac{1}{2}}

\newcommand{\CR}{\mathcal{R}}

\newcommand{\mc}{\mathcal{MC}}
\DeclareMathOperator{\eq}{eq}
\DeclareMathOperator{\coeq}{coeq}
\DeclareMathOperator{\Tot}{Tot}
\DeclareMathOperator{\Diag}{Diag}
\DeclareMathOperator{\CHarr}{CHarr}

\DeclareMathOperator{\gr}{gr}

\newcommand{\<}{\langle\,}
\renewcommand{\>}{\,\rangle}
\renewcommand{\o}{\otimes}

\newcommand{\B}{\mathsf{B}}

\DeclareMathOperator{\Aut}{Aut}
\DeclareMathOperator{\Hom}{Hom}
\DeclareMathOperator{\GL}{GL}
\DeclareMathOperator{\coker}{coker}
\DeclareMathOperator{\Spec}{Spec}

\begin{document}

\title[The derived Maurer-Cartan locus]{The Derived Maurer-Cartan Locus}

\author{Ezra Getzler}

\address{Department of Mathematics, Northwestern University, Evanston,
  Illinois, USA}

\email{getzler@northwestern.edu}

\thanks{The main results in this paper were obtained while the author
  held a Microsoft Research Visiting Fellowship at the Isaac Newton
  Institute of Mathematical Sciences of Cambridge University. He
  thanks the organizers of the programme on Operads and Multiple
  Zeta-Values, John D.S. Jones and Bruno Vallette, for this
  opportunity, and the participants in the programme, especially Jon
  Pridham, for valuable conversations on this subject. \\ \indent This
  work was completed while the author was a visitor at
  the University of Geneva. The author is partially supported by
  Travel Grant \#243025 of the Simons Foundation.}

\maketitle 

\begin{abstract}
  The derived Maurer-Cartan locus is a functor $\MC^\bull$ from
  differential graded Lie algebras to cosimplicial schemes. If $L$ is
  differential graded Lie algebra, let $L_+$ be the truncation of $L$
  in positive degrees $i>0$. We prove that the differential graded
  algebra of functions on the cosimplicial scheme $\MC^\bull(L)$ is
  quasi-isomorphic to the Chevalley-Eilenberg complex of $L_+$.
\end{abstract}

\section{Introduction}

Derived algebraic geometry is a non-linear analogue of homological
algebra. Just as homological algebra studies modules $M$ through
projective resolutions
\begin{equation*}
  \dots \to P_2 \to P_1 \to P_0 \to M \to 0 ,
\end{equation*}
derived algebraic geometry studies algebraic schemes through
resolutions by derived schemes. In this paper, we will only concern
ourselves with affine derived schemes.

In characteristic zero, derived schemes may be represented as
differential graded schemes or as cosimplicial schemes: differential
graded schemes were introduced by Tate~\cite{Tate}, and studied
further by Ciocan-Fontanine and Kapranov~\cite{CK}. The second
approach is largely due to Quillen \cite{Quillen}.

If $A^*$ is a differential graded algebra, denote by $A^\sharp$ its
underlying graded algebra. In this paper, we only consider
differential graded commutative algebras $A^*$ such that $A^i=0$ for
$i>0$. If $E^*$ is a vector space concentrated in negative degrees,
let $\Sym E$ be the free graded commutative algebra generated by
$E^*$: this is a polynomial algebra in generators in negative even
degrees tensored with an exterior algebra in generators in negative
odd degrees.

An affine differential graded scheme $\CX$ over a field $K$ of
characteristic zero is characterized by its ring of functions
$\CO(\CX)$, which is a differential graded commutative algebra over
$K$, with differential $d\colon\CO^*(\CX)\to\CO^{*+1}(\CX)$,
satisfying the following conditions:
\begin{enumerate}[a)]
\item $\CO^i(\CX)=0$ in positive degree $i>0$, and there is a regular
  affine variety $X$ such that $\CO^0(\CX)\cong\CO(X)$;
\item there is a graded vector bundle
  \begin{equation*}
    E^* = E^{-1} \oplus E^{-2} \oplus \dots
  \end{equation*}
  over $X$, and an isomorphism of graded commutative algebras
  \begin{equation*}
    \CO^*(\CX)^\sharp \cong \Gamma(X,\Sym E)
  \end{equation*}
  over $\CO(X)$.
\end{enumerate}

The condition that the underlying variety $X$ be regular is not
usually taken to be part of the definition: Ciocan-Fontanine and
Kapranov call differential graded schemes satisfying this additional
condition \textbf{differential graded manifolds}. But this condition
will always be satisfied in this paper. (In the language of
homotopical algebra, it is a fibrancy condition: it is analogous to
restricting attention to projective resolutions in homological
algebra.)

The \textbf{classical locus} $\pi^0(\CX)\subset X$ of a derived scheme
is the vanishing locus of the sheaf of ideals
\begin{equation*}
  \im\bigl( d\colon \CO^{-1}(\CX) \to \CO^0(\CX) \bigr) ,
\end{equation*}
or equivalently, the spectrum of the quotient ring
\begin{equation*}
  H^0(\CO^*(\CX),d) = \coker\bigl( d\colon \CO^{-1}(\CX) \to \CO^0(\CX) \bigr) .
\end{equation*}
In the special case that the cohomology of $\CO(\CX)$ is concentrated
in degree $0$, the differential graded scheme $\CX$ should be thought
of as a resolution of the vanishing locus $\pi^0(\CX)$, in the same
way as a projective resolution resolves a module.

Any regular affine scheme is a differential graded affine scheme, but
there are many more examples. Tate proved in \cite{Tate} that given
any finitely generated Noetherian commutative algebra $R$ over a field
$K$ of characteristic $0$, there is a differential graded commutative
ring $\CR$ of the above type such that $\CR^\sharp$ is a finitely
generated free graded commutative algebra and
\begin{equation*}
  H^i(\CR) \cong
  \begin{cases}
    R , & i=0 , \\ 0 , & i<0 .
  \end{cases}
\end{equation*}
In effect, $\CR$ is the ring of functions on an affine differential
graded scheme $\CX$ with $\pi^0(\CX)\cong\Spec(R)$. Tate also proved
that this affine differential graded scheme is essentially unique, in
the sense that given any two differential graded algebras $\CR_0$ and
$\CR_1$ with the above properties, there is morphism of differential
graded algebras from $\CR_0$ to $\CR_1$ such that the following
diagram commutes:
\begin{equation*}
  \begin{xy}
    \Dtriangle/{-->}`>`<-/<600,300>[\CR_0`R`\CR_1;``]
  \end{xy}
\end{equation*}

Let $L^*$ be a differential graded Lie algebra. This means that $L^*$
is a cochain complex, with differential $\delta\colon L^*\to L^{*+1}$,
with a bilinear bracket $[-,-]\colon L^i\times L^j\to L^{i+j}$, which
is graded symmetric,
\begin{equation*}
  [x,y] = - (-1)^{ij} [y,x] , \quad \text{$x\in L^i$, $y\in L^j$,}
\end{equation*}
satisfies the graded Jacobi identity,
\begin{equation*}
  [x,[y,z]] = [[x,y],z] + (-1)^{ij} [y,[x,z]] , \quad \text{$x\in L^i$, $y\in
    L^j$, $z \in L^k$,}
\end{equation*}
and the Leibniz identity,
\begin{equation*}
  \delta[x,y] = [\delta x,y] + (-1)^i [x,\delta y] , \quad \text{$x\in L^i$,
  $y\in L^j$.}
\end{equation*}
The differential graded Lie algebra is of \textbf{finite type} if
$\dim L^i<\infty$ for all $i$, and vanishes for $i\ll0$.

The function
\begin{equation*}
  F(\mu) = \delta\mu + \half [\mu,\mu]
\end{equation*}
from $L^1$ to $L^2$ is called the \textbf{curvature}. It satisfies the
\textbf{Bianchi identity}
\begin{equation}
  \label{Bianchi}
  \delta F(\mu) + [\mu,F(\mu)] = 0 .
\end{equation}
The \textbf{Maurer-Cartan locus} $\MC(L)\subset L^1$ of $L^*$ is the
vanishing locus of the curvature $F(\mu)=0$.

There is a variant of the Maurer-Cartan locus, called the
\textbf{Deligne groupoid}, which takes into account the component
$L^0$ of the differential graded Lie algebra in degree $0$. The Lie
algebra $L^0$ acts on $L^1$ by vector fields $X_\xi$, $\xi\in L^0$,
given by the formula
\begin{equation*}
  X_\xi(\mu) = - \delta\xi - [\mu,\xi] , \quad \mu \in L^1 .
\end{equation*}
Let $G^0$ be the universal algebraic group with Lie algebra $L^0$ (so
that all finite-dimensional representations of $L^0$ come from a
representation of $G^0$), and suppose that the above action of $L^0$
exponentiates to an action of $G^0$ on $L^1$: for example, this will
be the case when the differential $\delta\colon L^0\to L^1$
vanishes. Then this action preserved the Maurer-Cartan locus $\MC(L)$:
the groupoid associated to the action of $G^0$ on $\MC(L)$ is called
the Deligne groupoid of the differential graded Lie algebra
$L^*$. Locally, most, if not all, deformation problems in algebraic
geometry may be represented as Deligne groupoids.

\begin{example}
  \label{quot}

  Let $R$ be a commutative ring, and let $A^*$ be a differential
  graded algebra defined over $R$. The Hochschild complex $\B^*(R,A)$
  is the bigraded abelian group
  \begin{equation*}
    \B^{j,k}(R,A) = \Hom(R^{\o j},A^k)
  \end{equation*}
  with differentials $d\colon \B^{j,k}\to \B^{j,k+1}$ and
  $\delta\colon \B^{j,k}\to \B^{j+1,k}$, given by the formulas
  \begin{equation*}
    (dc)(r_1,\dots,r_{j+1}) = r_1c(r_2,\dots,r_{j+1}) +
    \sum_{i=1}^j (-1)^i c(r_1,\dots,r_ir_{i+1},\dots,r_{j+1})
  \end{equation*}
  and $(\delta c)(r_1,\dots,r_j) = \delta(c(r_1,\dots,r_j))$.
  Furthermore, $\B(R,A)$ is a graded Lie algebra, with bracket
  \begin{multline*}
    [c_1,c_2](r_1,\dots,r_{j_1+j_2}) = (-1)^{j_1k_2} \,
    c_1(r_1,\dots,r_{j_1})c_2(r_{j_1+1},\dots,r_{j_1+j_2}) \\
    - (-1)^{(j_1+k_1)(j_2+k_2)+j_2k_1} \,
    c_2(r_1,\dots,r_{j_2})c_1(r_{j_2+1},\dots,r_{j_1+j_2}) ,
  \end{multline*}
  where $c_1\in \B^{j_1,k_1}$ and $c_2\in \B^{j_2,k_2}$.

  Special cases of this construction give differential graded Lie
  algebras with applications to deformation theory. For example, let
  $M$ be a finite-dimensional vector space and let $n$ be a natural
  number, and consider the graded algebra
  \begin{equation*}
    A^k =
    \begin{cases}
      \End(M) , & k=0 , \\
      \Hom(R^{\oplus n},M) , & k=1 , \\
      0 , & \text{otherwise.}
    \end{cases}
  \end{equation*}
  The product on $A^*$ is given by the product on $\End(M)$, and the
  natural pairing
  $\End(M)\o\Hom(R^{\oplus n},M)\to\Hom(R^{\oplus n},M)$, and
  otherwise it vanishes. The differential on $A^*$ is zero. A
  Maurer-Cartan element of $\B(R,A)$ is a pair $(\rho,f)$, where
  $\rho$ is an action of $R$ on $M$, and $f\colon R^{\oplus n}\to M$
  is a morphism of $R$-modules. The Deligne groupoid is given by the
  natural action of the semisimple algebraic group $\GL(M)$ on
  $\MC(\B(R,A))$, which has the effect of conjugating $\rho$, and
  composing with $f$.

  The \emph{Quot} scheme of projective geometry is obtained by an
  analogue of this construction: one takes a finite dimensional
  truncation $R$ of the homogenous ring of polynomials
  $\C[x_0,\dots,x_N]$, a finite dimensional truncation $M$ of a
  homogenous $\C[x_0,\dots,x_N]$-module, and forms the Lie subalgebra
  $\B_0(R,A)\subset\B(R,A)$ consisting of elements of zero total
  homogeneity. For further details, see \cite{CK}.
  
\end{example}

\begin{example}
  
  Our second example is the Harrison complex of a vector space
  $R$. Given natural numbers $p$ and $q$, let $\Sh(p,q)$ be the set of
  partitions of $\{0,\dots,p+q-1\}$ into disjoint subsets
  $I=(i_1<\dots<i_p)$ and $J=(j_1<\dots<j_q)$. Harrison cochains are
  multilinear maps from $R$ to itself which vanish on shuffles:
  $\CHarr^k(R,R)$ is the set of $c\in \Hom(R^{\o k+1},R)$ such that
  for all $0<p<k$, we have
  \begin{equation*}
    \sum_{(I,J)\in\Sh(p,k-p)} (-1)^{i_1+\dots+i_p} \,
    c(r_{i_1},\dots,r_{i_p},r_{j_1},\dots,r_{j_{k-p}}) = 0 .
  \end{equation*}
  This is a graded Lie algebra with respect to the Gerstenhaber
  bracket: if $c_1\in\CHarr^{k_1}(R,R)$ and $c_2\in\CHarr^{k_2}(R,R)$,
  the bracket equals
  \begin{multline*}
    [c_1,c_2](r_0,\dots,r_{k_1+k_2}) = \sum_{i=0}^{k_1} (-1)^{ik_2}
    c_1(r_0,\dots,c_2(r_i,\dots,r_{i+k_2}),\dots,r_{k_1+k_2}) \\
    - \sum_{i=0}^{k_2} (-1)^{k_1k_2+ik_1}
    c_2(r_0,\dots,c_1(r_i,\dots,r_{i+k_1}),\dots,r_{k_1+k_2}) .
  \end{multline*}
  The Maurer-Cartan locus $\MC(\CHarr(R,R))$ consists of all bilinear
  maps $\mu\in\CHarr^1(R,R)$ such that $[\mu,\mu]=0$. This is the
  space of all commutative associative products on $R$.

  The Lie algebra $\CHarr^0(R,R)$ may be identified with $\End(R)$,
  which is a semisimple Lie algebra with associated universal
  algebraic group $\GL(R)$. This group acts on the graded Lie algebra
  $\CHarr(R,R)$ by the formula
  \begin{equation*}
    (g\cdot c)(r_0,\dots,r_k) = g(c(g^{-1}(r_0),\dots,g^{-1}(r_k))) .
  \end{equation*}
  Thus, the Deligne groupoid of $\CHarr^*(R,R)$ is the space of
  commutative associative products on $R$ up to conjugation.

\end{example}

\begin{example}
  
  As a final example, we sketch an application of this formalism in
  the study of holomorphic vector bundles. This example lies outside
  algebraic geometry, and requires the use of Sobolev spaces to make
  any sense of it.

  Given a complex manifold $X$ and a holomorphic vector bundle $E$ on
  $X$, the Dolbeault complex $A^{0,*}(X,\End(E))$ is a differential
  graded Lie algebra, with differential $\overline{\partial}$. The
  curvature of an element $\mu\in A^{0,1}(X,\End(E))$ is the
  obstruction in $A^{0,2}(X,\End(E))$ to the first-order differential
  operator $\overline{\partial}+\ad(\mu)$ inducing a holomorphic
  structure on $E$. Thus, the Maurer-Cartan locus of
  $A^{0,*}(X,\End(E))$ is the space of holomorphic structures (or
  Cauchy-Riemann operators) on $E$.

  The graded Lie algebra $A^{0,0}(X,\End(E))$ is the space of sections
  of the endomorphism bundle $\End(E)$. The associated group is the
  gauge group of $E$, which is the space of sections of the smooth
  bundle of Lie groups $\Aut(G)$, and the Deligne groupoid models the
  stack of holomorphic structures on $E$ up to gauge equivalence.

\end{example}

The \textbf{differential graded Maurer-Cartan locus} $\mc(L)$ of a
differential graded Lie algebra $L^*$ of finite type is the affine
differential graded scheme with underlying scheme the affine space
$L^1$, and with the graded algebra of functions
\begin{equation*}
  \CO(\mc(L))^\sharp = \Sym (L_+[1]{}^\vee) .
\end{equation*}
Here, $L^*_+$ is the truncation of $L$ in positive degrees:
\begin{equation*}
  L^i_+ =
  \begin{cases}
    L^i , & i\ge1 , \\
    0 , & i<1 ,
  \end{cases}
\end{equation*}
and $L_+[1]$ denotes the shift of the cochain complex $L_+$ downward
in degree by $1$. This graded algebra may be identified with the
graded vector space of Chevalley-Eilenberg cochains of the
differential graded Lie algebra $L^*_+$. The differential $d$ on
$\CO(\mc(L))$ is the differential of the Chevalley-Eilenberg complex:
it is the sum of the adjoints of $\delta$, which maps $(L^{i+1})^\vee$
to $(L^i)^\vee$, and of $[-,-]$, which maps $(L^k)^\vee$ to
\begin{equation*}
  \bigoplus_{i=1}^{k-1} \, (L^i)^\vee \o (L^{k-i})^\vee .
\end{equation*}

The image of the differential $d$ in $\CO^0(\mc(L))$ is the ideal
generated by the curvature $F(x)$. This proves the following result.
\begin{proposition}
  The classical locus $\pi^0(\mc(L))$ of the differential graded
  scheme $\mc(L)$ is the Maurer-Cartan locus $\MC(L)$.
\end{proposition}

There is also a differential graded analogue of the Deligne
groupoid. For simplicity, we consider only the case where the
differential $\delta\colon L^0\to L^1$ vanishes. The universal
algebraic group $G^0$ with Lie algebra $L^0$ acts on the differential
graded Lie algebra $L_+$, and hence on the derived Maurer-Cartan locus
$\MC(L)$. The groupoid in differential graded schemes associated to
this action is the derived Deligne groupoid of $L$. For example, the
derived Deligne groupoid of the differential graded Lie algebra
$\B_0(R,A)$ of Example~\ref{quot} gives rise to the derived
\emph{Quot} scheme of Ciocan-Fontanine and Kapranov \cite{CK}, while
the derived Deligne groupoid of the graded Lie algebra $\CHarr(R,R)$
leads to the derived stack of commutative associative products on $R$.

Quillen~\cite{Quillen} introduced cosimplicial schemes as an alternate
foundation for the theory of derived geometry: unlike differential
graded schemes, they give the correct model for derived schemes even
in positive characteristic (though in this paper, we will only
consider cosimplicial schemes in characteristic zero). A cosimplicial
scheme $X^\bull$ is a functor from the category $\Delta$ of nonempty
finite totally ordered sets to the category of schemes. For $n\ge0$,
denote the object
\begin{equation*}
  0 < \dots < n
\end{equation*}
of $\Delta$ by $[n]$: the functor $X^\bull$ takes the value $X^n$ at $[n]$.

A cosimplicial scheme $X^\bull$ is the spectrum of a simplicial
commutative ring. Quillen proved that in characteristic zero,
simplicial commutative algebras and differential graded commutative
algebras have equivalent homotopy theory, in the following sense: the
normalization functor $N_*$ from simplicial vector spaces to chain
complexes induces a functor from simplicial commutative algebras to
differential graded commutative algebras, also denoted $N_*$, and this
functor induces an equivalence of homotopy categories. (In fact, $N_*$
is a right Quillen equivalence with respect to the projective closed
model structures on these categories; cf.\ \cite{RH}*{Section~4}.) We
review the construction of the functor $N_*$ in Sections~2 and 3.

In the category $\Delta$, we have the \textbf{coface maps}
\begin{equation*}
  d^i \colon [n-1] \to [n] , \quad 0\le i\le n ,
\end{equation*}
defined by
\begin{equation*}
  d^i(j) =
  \begin{cases}
    j , & j<i , \\
    j+1 , & j\ge i ,
  \end{cases}
\end{equation*}
and the \textbf{codegeneracy maps}
\begin{equation*}
  s^i \colon [n+1] \to [n] , \quad 0\le i\le n ,
\end{equation*}
defined by
\begin{equation*}
  s^i(j) =
  \begin{cases}
    j , & j<i , \\
    j-1 , & j\ge i .
  \end{cases}
\end{equation*}
If $X_\bull$ is a cosimplicial object, we denote the induced morphisms
$d^i\colon X^{n-1}\to X^n$ and $s^i\colon X^{n+1}\to X^n$ by the same
symbol. If $X_\bull$ is a simplicial object, we denote the face and
degeneracy morphisms by $\p_i\colon X^n \to X^{n-1}$ and
$\sigma_i\colon X^n\to X^{n+1}$.

The maximal augmentation of a cosimplicial scheme is the equalizer
\begin{equation*}
  \pi^0(X^\bull) = \eq \Bigl(
  \begin{xy}
    \morphism|a|/{@{>}@<3pt>}/<400,0>[X^0`X^1;d^0]
    \morphism|b|/{@{>}@<-3pt>}/<400,0>[X^0`X^1;d^1]
  \end{xy} \Bigr)
\end{equation*}
Observe the analogy with the definition of the set of components of a
simplicial set $X_\bull$, which is the coequalizer
\begin{equation*}
  \pi_0(X_\bull) = \coeq \Bigl(
  \begin{xy}
    \morphism|a|/{@{>}@<3pt>}/<400,0>[X_1`X_0;\p_0]
    \morphism|b|/{@{>}@<-3pt>}/<400,0>[X_1`X_0;\p_1]
  \end{xy} \Bigr)
\end{equation*}

It is the goal of this paper to make the equivalence between
differential graded schemes and cosimplicial schemes in characteristic
zero as explicit as possible for derived Maurer-Cartan loci. The
realization of the derived Maurer-Cartan locus as a cosimplicial
scheme that we propose is new. This realization may also be used in
other settings, for example when afine schemes are replaced by Banach
analytic spaces: in contrast with the differential graded
Maurer-Cartan locus, its definition does not require making sense of
the Chevalley-Eilenberg complex for differential graded Banach Lie
algebras.

In order to realize the derived Maurer-Cartan locus as a cosimplicial
scheme, we introduce a certain cosimplicial differential graded
commutative algebra $\Lambda^\bull$. As a graded algebra, $\Lambda^n$
is the exterior algebra generated by elements $\{e_0,\dots,e_n\}$ in
degree $-1$: the differential on $\Lambda^n$ is defined on the
generators $e_i$ by $\delta e_i=1$. A morphism $f\colon[m]\to[n]$ of
$\Delta$ induces a homomorphism $f\colon\Lambda^m\to\Lambda^n$ of
differential graded commutative algebras by its action on the
generators: $f(e_i)=e_{f(i)}$.

The tensor product $L^*\o\Lambda^n$ of a differential graded Lie
algera $L^*$ with the differential graded commutative algebra
$\Lambda^n$ is again a differential graded Lie algebra, with bracket
\begin{equation*}
  [ x_1\o\alpha_1 , x_2\o\alpha_2 ] = (-1)^{j_2k_1} \, [x_1,x_2]\o
  \alpha_1\alpha_2 ,
\end{equation*}
where $x_1\in L^{j_1}$, $x_2\in L^{j_2}$,
$\alpha_1 \in (\Lambda^n){}^{k_1}$, and
$\alpha_2 \in (\Lambda^n){}^{k_2}$.

\begin{definition}
  \label{dMC}
  The \textbf{derived Maurer-Cartan locus} $\MC^\bull(L)$ of a
  differential graded Lie algebra $L$ is the cosimplicial scheme
  \begin{equation*}
    \MC^n(L) = \MC(L\o\Lambda^n) .
  \end{equation*}
\end{definition}

We may now state our main result.
\begin{theorem}
  \label{main}
  The normalization $N_*(\CO(\MC^\bull(L)))$ of the simplicial
  commutative algebra $\CO(\MC^\bull(L))$ of functions on the
  cosimplicial scheme $\MC^\bull(L)$ is a differential graded
  commutative algebra concentrated in nonpositive degrees. There is a
  natural homomorphism of differential graded commutative algebras
  \begin{equation*}
    \Phi \colon \CO(\mc(L)) \too N_{-*}(\CO(\MC^\bull(L))) ,
  \end{equation*}
  which is a quasi-isomorphism.
\end{theorem}

Thus, the two realizations of the derived Maurer-Cartan locus of $L^*$
are equivalent.

In Section~2 of this paper, we review the Dold-Kan correspondence
between simplicial abelian groups and connective chain complexes.

In Section~3, we review Eilenberg and Mac\,Lane's formulation of the
Eilenberg-Zilber theorem, in particular, the formulas for the
Alexander-Whitney and shuffle maps. We show that for \emph{abelian}
differential graded Lie algebras $L^*$, the derived Maurer-Cartan
locus $\MC^\bull(L)$ may be identified with the cosimplicial vector
space $K^\bull(L_+[1])$ associated to the coconnective cochain complex
$L^*_+[1]$. In this sense, $\MC^\bull(L)$ is a nonlinear
generalization of the functor $K^\bull$ realizing the equivalence of
Dold and Kan between the categories of coconnective cochain complexes
and cosimplicial vector spaces.

A key idea in the proof of Theorem~\ref{main} is the observation that
the derived Maurer-Cartan locus is a \textbf{grouplike} cosimplicial
scheme, in the sense of Bousfield and Kan~\cite{BK}: although
$\MC^\bull(L)$ is not actually a cosimplicial group scheme unless
$L^*$ is abelian, it is close to being so in a certain precise sense,
as we explain in Section~4.

In particular, the underlying graded commutative algebra
$N_*(\CO(\MC^\bull(L)))^\sharp$ of the differential graded commutative
algebra $N_*(\CO(\MC^\bull(L)))$ only depends on the graded vector
space $L^\sharp$ underlying $L^*$. We also prove that
$N_*(\CO(\MC^\bull(L)))^\sharp$ is a free graded commutative algebra:
the proof uses a result of Milnor and Moore \cite{MM}*{Theorem~7.5}
(which they ascribe to Leray), and a recent important complement to
the Eilenberg-Zilber Theorem due to \v{S}evera and Willwacher
\cite{SW} and Aguiar and Mahajan \cite{AM}.

In Section~5, we complete the proof of Theorem~\ref{main}. In
Section~6, we state the generalizion of our results for nilpotent
\Linf-algebras.

\section{The Dold-Kan correspondence for cosimplicial abelian groups}

The \textbf{normalized chain complex} of a simplicial abelian group
$A_\bull$ is the graded abelian group
\begin{equation*}
  N_n(A) = A_n \Bigm/ \sum_{i=0}^{n-1} \im\bigl( \sigma_i
  \colon A_{n-1} \to A_n \bigr) ,
\end{equation*}
with differential
\begin{equation*}
  \p = \sum_{i=0}^n (-1)^i \p_i \colon N_n(A) \to N_{n-1}(A) .
\end{equation*}
The chain complex $N_*(A)$ is \textbf{connective}: it vanishes in
negative homological, or positive cohomological, degrees. (We may
consider any chain complex $V_*$ to be a cochain complex $V^*$, by
setting $V^*=V_{-*}$.)

For example, the abelian group $N_k(\Z\Delta^n)$ is a free abelian
group with generators
\begin{equation*}
  \{x_{i_0\dots i_k} \mid 0\le i_0<\dots<i_k\le n\} ,
\end{equation*}
where $x_{i_0\dots i_k}$ corresponds to the nondgenerate simplex
$[k]\to[n]$ with vertices
\begin{equation*}
  \{i_0,\dots,i_k\}\subset\{0,\dots,n\} .
\end{equation*}
The differential $\p\colon N_k(\Z\Delta^n)\to N_{k-1}(\Z\Delta^n)$ is
given by the formula
\begin{equation*}
  \p x_{i_0\dots i_k} = \sum_{j=0}^k (-1)^j
  x_{i_0\dots\widehat{\imath}{}_j\dots i_k} .
\end{equation*}

The right-adjoint of the functor $N_*$ is the functor $K_\bull$ from
chain complexes to simplicial abelian groups defined by Eilenberg and
Mac\,Lane~\cite{EM2}:
\begin{equation*}
  \Map(N_*(A),Z_*) \cong \Map(A_\bull,K_\bull(Z)) .
\end{equation*}
The Yoneda lemma implies that the $n$-simplices of the simplicial
abelian group $K(Z)$ are given by the formula
\begin{equation}
  \label{sYoneda}
  K_n(Z) \cong \Map(N_*(\Z\Delta^n),Z_*) .
\end{equation}
Dold~\cite{Dold} and Kan~\cite{Kan} proved that the adjoint pair of
functors $N \dashv K$ yields an adjoint equivalence between the
categories of simplicial abelian groups and connective chain
complexes. Dold and Puppe~\cite{DP} extended this equivalence to
arbitrary abelian categories, with the functors $N_*$ and $K_\bull$
being given by the same formulas as in the category of abelian groups.

The opposite category to the category of abelian groups is an abelian
category: the corresponding categories of connective chain complexes
and simplicial objects are the categories of coconnective chocain
complexes (cochain complexes vanishing in negative degree) and
cosimplicial abelian groups. Let us make the adjoint equivalence of
Dold and Puppe more explicit in this situation.

The normalized cochain complex of a cosimplicial abelian group
$A^\bull$ is the graded abelian group
\begin{equation*}
  N^n(A) = \bigcap_{i=0}^{n-1} \, \ker \bigl( s^i \colon A^n \to
  A^{n-1} \bigr) ,
\end{equation*}
with differential
\begin{equation*}
  \sum_{i=0}^{n+1} (-1)^i d^i \colon N^n(A) \to N^{n+1}(A) .
\end{equation*}
The functor $N^*$ has a left-adjoint $K^\bull$, which takes cochain
complexes to cosimplicial abelian groups. In fact, since it is an
equivalence, it is also the right-adjoint of $N^*$. Let $\Delta_n$ be
the cosimplicial set corepresented by the object $[n]\in\Delta$: we
have
\begin{equation*}
  (\Delta_n){}^k = \Delta([n],[k]) .
\end{equation*}
Form the cosimplicial abelian group $\Z\Delta_n^\bull$.  If $Z^\bull$
is a cosimplicial abelian group, Yoneda's Lemma implies that
\begin{equation*}
  Z^n \cong \Map(\Z\Delta^\bull_n,Z^\bull) .
\end{equation*}
Thus, parallel to the case of simplicial abelian groups
\eqref{sYoneda}, we see that
\begin{align}
  \label{K}
  K^n(Z) &\cong \Map(\Z\Delta^\bull_n,K^\bull(Z)) \\
  \notag &\cong \Map(N^*(\Z\Delta_n),Z^*) .
\end{align}

Unlike its cousin $N_*(\Z\Delta^n)$, the cochain complex
$N^*(\Z\Delta_n)$ has not been discussed in the literature. Let
$f_{n_0\dots n_k}\colon[n]\to[k]$ be the morphism such that
\begin{equation*}
  f_{n_0\dots n_k}^{-1}(j) = \{ n_0+\dots+n_{j-1},\dots,n_0+\dots+n_j-1 \} .
\end{equation*}
The action of the coface maps is given by
\begin{equation}
  \label{coface}
  d^if_{n_0\dots n_k} = f_{n_0\dots n_{i-1}0n_i\dots n_k} ,
\end{equation}
and the action of the codegeneracy maps by
\begin{equation}
  \label{codegeneracy}
  s^if_{n_0\dots n_k} = f_{n_0\dots n_i+n_{i+1}\dots n_k} .
\end{equation}
Denote by $[f_{n_0\dots n_k}]\in N^k(\Z\Delta_n)$ the image of
$f_{n_0\dots n_k}$ in the normalized cochain complex.

The chain complex $N_*(A)$ may be represented as a colimit.
\begin{lemma}[Dold~\cite{Dold}, Lemma 1.11]
  \label{dold}
  The quotient map from $A_n$ to $N_n(A)$ induces a natural
  isomorphism of abelian groups
  \begin{equation*}
    N_n(A) \cong \bigcap_{i=1}^n \, \ker\bigl( \p_i \colon A_n
    \to A_{n-1} \bigr) .
  \end{equation*}
  Under this isomorphism, the differential $\p$ corresponds to the
  operator
  \begin{equation*}
    \p_0 \colon \bigcap_{i=1}^n \, \ker\bigl( \p_i
    \colon A_n \to A_{n-1} \bigr) \to \bigcap_{i=1}^{n-1} \,
    \ker\bigl( \p_i \colon A_{n-1} \to A_{n-2} \bigr) .
  \end{equation*}
\end{lemma}

In the opposite category to the category of abelian groups, this lemma
yields the following corollary.
\begin{corollary}
  The abelian group $N^k(\Z\Delta_n)$ is a free abelian group,
  generated by the elements $[f_{n_0\dots n_k}]$, where
  $n_0+\dots+n_k=n+1$ and $n_i>0$ when $i<n$. The differential is
  given by the formula
  \begin{equation*}
    d[f_{n_0\dots n_k}] = [f_{n_0\dots n_k0}] .
  \end{equation*}
\end{corollary}

\section{The Eilenberg-Zilber theorem}

Let $A_{\bull\bull}$ be a bisimplicial abelian group: a contravariant
functor from the category $\Delta\times\Delta$ to the category of
abelian groups.  Denote the maps defining the first simplicial
structure by $\p_i^{(1)}$ and $\sigma_i^{(1)}$, and those defining the
second simplicial structure by $\p_i^{(2)}$ and $\sigma_i^{(2)}$.

By the Dold-Kan theorem, the categories of bisimplicial abelian groups
and first-quadrant double complexes $X_{**}$ are equivalent. This
equivalence is realized by the naturally equivalent functors
$N_*^{(2)}N_*^{(1)}\cong N_*^{(1)}N_*^{(2)}$. Denote either of these
functors by $N_{**}$.

The double complex $N_{**}(A)$ has two commuting differentials
$\p^{(1)}$ and $\p^{(2)}$, of bidegree $(1,0)$ and $(0,1)$
respectively. The \textbf{total} chain complex of this double complex
is the chain complex
\begin{equation*}
  \Tot_k(N_{**}(A)) = \bigoplus_{p+q=k} \, N_{pq}(A) ,
\end{equation*}
with differential $\p=\p^{(1)}+(-1)^p\p^{(2)}$.

The diagonal of a bisimplicial abelian group is the simplicial abelian
group
\begin{equation*}
  \Diag_p(A) = A_{pp} .
\end{equation*}
The Eilenberg-Zilber theorem~\cite{EZ} compares the chain complex
$\Tot_*(N_{**}(A))$ to the normalization $N_*(\Diag_\bull(A))$ of the
diagonal of $A_{\bull\bull}$. We will use the following explicit
formulation of the theorem.
\begin{theorem}[Eilenberg and Mac\,Lane~\cite{EM2}, Section~2]
  There are natural morphisms of complexes
  \begin{equation*}
    f \colon N_*(\Diag_\bull(A)) \too \Tot_*(N_{**}(A))
  \end{equation*}
  and
  \begin{equation*}
    g \colon \Tot_*(N_{**}(A)) \too N_*(\Diag_\bull(A))
  \end{equation*}
  and a natural homotopy
  \begin{equation*}
    h \colon N_*(\Diag_\bull(A)) \too N_{*+1}(\Diag_\bull(A)) ,
  \end{equation*}
  such that $fg$ is the identity of $\Tot_*(N_{**}(A))$, $gf+\p h+h\p$
  is the identity of $N_*(\Diag_\bull(A))$, and $fh$ and $hg$
  vanish. In particular, the homology groups of the complexes
  $\Tot_*(N_{**}(A))$ and $N_*(\Diag_\bull(A))$ are isomorphic.
\end{theorem}

The explicit formulas for the natural transformations $f$ and $g$ are
as follows. The map $f$ from $N_k(\Diag_\bull(A))$ to
$\Tot_k(N_{**}(A))$ is the \textbf{Alexander-Whitney} map
\begin{equation}
  \label{AW}
  f_k = \sum_{p=0}^k \p^{(1)}_{p+1}\dots\p^{(1)}_k
  \p^{(2)}_0\dots\p^{(2)}_{p-1} .
\end{equation}

The component $g_{pq}$ of $g$ mapping $A_{pq}$ to
$N_{p+q}(\Diag_\bull(A))$ is given by the formula (Eilenberg and
Mac\,Lane~\cite{EM1}, Section~5)
\begin{equation}
  \label{shuffle}
  g_{pq} = \sum_{\{i_1<\dots<i_p\}\coprod\{j_1<\dots<j_q\}\in\Sh(p,q)}
  (-1)^{\sum_{\ell=1}^p(i_\ell-\ell+1)} \,
  \sigma^{(1)}_{j_q}\dots\sigma^{(1)}_{j_1}
  \sigma^{(2)}_{i_p}\dots\sigma^{(2)}_{i_1} .
\end{equation}
This map is called the \textbf{shuffle} map.

A \textbf{simplicial coalgebra} is a simplicial $R$-module $A_\bull$
together with simplicial morphisms
$c\colon A_\bull\to A_\bull\o A_\bull$, the comultiplication, and
$\eps\colon A_\bull\to R$, the augmentation, such that the diagram
\begin{equation*}
  \begin{xy}
    \Square[A_\bull`A_\bull\o A_\bull`A_\bull\o A_\bull`A_\bull\o
    A_\bull\o A_\bull;c`c`c\o A`A\o c]
  \end{xy}
\end{equation*}
commutes (coassociativity), and both $(\eps\o A)c$ and $(A\o\eps)c$
equal the identity morphism of $A_\bull$.

Using the Alexander-Whitney map, we may show that the normalized chain
complex of a simplicial coalgebra over a commutative ring $R$ is a
differential graded coalgebra. Let $(A\boxtimes A)_{\bull\bull}$ be
the bisimplicial $R$-module
\begin{equation*}
  (A\boxtimes A)_{pq} = A_p \o A_q .
\end{equation*}
In particular, $\Diag_\bull(A\boxtimes A)\cong A_\bull\o A_\bull$ and
\begin{equation*}
  \Tot_*(N_{**}(A\boxtimes A)) \cong N_*(A) \o N_*(A) .
\end{equation*}
The comultiplication $c\colon A_\bull\to A_\bull\o A_\bull$ gives
a morphism of complexes
\begin{equation*}
  N_*(c) \colon N_*(A) \too N_*(\Diag_\bull(A\o A)) .
\end{equation*}
Composing with the Alexander-Whitney map, we obtain a map
\begin{equation*}
  fN_*(c) \colon N_*(A) \too N_*(A) \o N_*(A) .
\end{equation*}
It is easily checked that this morphism of chain complexes is
coassociative, and has
\begin{equation*}
  N_*(\eps) \colon N_*(A) \to N_*(R) \cong R
\end{equation*}
as a counit.

A \textbf{simplicial algebra} is a simplicial $R$-module $A_\bull$
together with simplicial morphisms
$m \colon A_\bull\o A_\bull\to A_\bull$, the multiplication, and
$\eta \colon R\to A_\bull$, the unit, such that the diagram
\begin{equation*}
  \begin{xy}
    \Square[A_\bull\o A_\bull\o A_\bull`A_\bull\o A_\bull`A_\bull\o
    A_\bull`A;m\o A`A\o m`m`m]
  \end{xy}
\end{equation*}
commutes (associativity), and both $m(\eta\o A)$ and $(A\o\eta)m$
equal the identity morphism of $A_\bull$.

Using the shuffle map, we may show that the normalized chain complex
of a simplicial algebra $A_\bull$ over a commutative ring $R$ is a
differential graded algebra. The multiplication
$m\colon A_\bull\o A_\bull\to A_\bull$ gives a morphism of complexes
\begin{equation*}
  N_*(m) \colon N_*(\Diag_\bull(A\boxtimes A)) \too N_*(A) .
\end{equation*}
Composing with the shuffle map, we obtain a morphism
\begin{equation*}
  N_*(m)g \colon N_*(A) \o N_*(A) \too N_*(A) .
\end{equation*}
This morphism of chain complexes is associative, and has
\begin{equation*}
  N_*(\eta)\colon N_*(R)\cong R\to N_*(A)
\end{equation*}
as a unit. In fact, more is true: if $A_\bull$ is a simplicial
\emph{commutative} algebra, then $N_*(A)$ is a differential graded
commutative algebra.

Parallel constructions in the opposite category to the category of
$R$-modules shows that the normalized cochain complex $N^*(A)$ of a
cosimplicial algebra $A^\bull$ is a differential graded algebra, and
that the normalized cochain complex $N^*(A)$ of a cosimplicial
(cocommutative) coalgebra $A^\bull$ is a differential graded
(cocommutative) coalgebra.

A \textbf{simplicial bialgebra} is a simplicial algebra $A_\bull$
which is at the same time a simplicial coalgebra, in such a way that
the comultiplication $c\colon A_\bull\to A_\bull\o A_\bull$ and
augmentation $\eps\colon A_\bull\to R$ are morphisms of simplicial
algebras. The following result is proved in Appendix~A of \v{S}evera
and Willwacher~\cite{SW} and Section~5.4 of Aguiar and Mahajan
\cite{AM}: the proof is by an explicit calculation verifying the
required compatibility between the Alexander-Whitney and shuffle maps.
\begin{proposition}
  \label{SW}
   The normalized chain complex $N_*(A)$ of a simplicial (commutative)
   bialgebra $A_\bull$ is a differential graded (commutative) bialgebra.
\end{proposition}

If $X_\bull$ is a simplicial set, the simplicial abelian group
$\Z X_\bull$ is a simplicial coalgebra, and the Alexander-Whitney map
makes $N_*(\Z X)$, the simplicial chain complex of $X_\bull$, into a
differential graded coalgebra. On the other hand, if $X^\bull$ is a
cosimplicial set, the cosimplicial abelian group $\Z X^\bull$ is a
cosimplicial cocommutative coalgebra, and the shuffle map makes
$N^*(\Z X)$ into a differential graded cocommutative coalgebra. In the
following proposition, we analyse the differential graded coalgebra
$N^*(\Z\Delta^\bull_n)$.
\begin{proposition}
  \label{Lambda}
  The dual $N^*(\Z\Delta_n)^\vee$ of the differential graded coalgebra
  $N^*(\Z\Delta_n)$ is isomorphic to the differential graded
  commutative algebra $\Lambda^n$. This duality is induced by the
  following pairing between the free abelian groups $N^1(\Z\Delta_n)$
  and $(\Lambda^n)^{-1}$:
  \begin{equation*}
    \< [f_{i,n-i+1}] , e_j \> =
    \begin{cases}
      1 , & i\le j , \\
      0 , & i>j .
    \end{cases}
  \end{equation*}
\end{proposition}
\begin{proof}
  If $\varphi\colon [m]\to[n]$ is a morphism of $\Delta$, then
  \begin{equation*}
    \varphi^*[f_{i,n-i+1}]=[f_{i',m-i'+1}] ,
  \end{equation*}
  where $i'$ is the cardinality of the set
  $\varphi^{-1}(\{0,\dots,i-1\})$. It is easily seen that
  \begin{equation*}
    \< \phi^*[f_{i,n-i+1}] , e_j \> = \< [f_{i,n-i+1}] ,
    e_{\varphi(j)} \> ,
  \end{equation*}
  and hence that the pairing between the simplicial abelian group
  $[n]\mapsto N^1(\Z\Delta_n)$ and the cosimplicial abelian group
  $[n]\mapsto(\Lambda^n){}^{-1}$ is compatible with the respective
  actions of the category $\Delta$. That is, the pairing descends to
  the colimit $N^1(\Z\Delta_\bull) \o_\Delta (\Lambda^\bull){}^{-1}$.
  
  Given a coalgebra $A$ with comultiplication $c\colon A\to A\o A$,
  let
  \begin{equation*}
    c^{(k)} = (A^{\o k-2}\o c)\dots(A\o c)c \colon A \too A^{\o k}
  \end{equation*}
  be the iterated coproduct. Let $p$ be the projection from
  $N^*(\Z\Delta_n)^{\o k}$ to $N^1(\Z\Delta_n)^{\o k}$. Let $\pi$ be
  the symmetrization operator
  \begin{equation*}
    \pi = \sum_{\sigma\in S_n} (-1)^{\eps(\sigma)} \sigma \colon
    N^1(\Z\Delta_n)^{\o k}\to N^1(\Z\Delta_n)^{\o k} .
  \end{equation*}
  The proposition is a consequence of the following formula:
  \begin{align*}
    pc^{(k)}[f_{n_0\dots n_k}] &= \pi\bigl( [f_{n_0,n_1+\dots+n_k}] \o
      [f_{n_0+n_1,n_2+\dots+n_k}] \o \dots \o
      [f_{n_0+\dots+n_{k-1},n_k}] \bigr) \\
    &\in N^1(\Z\Delta_n)^{\o k} .
  \end{align*}
  This formula is proved using the explicit formulas \eqref{shuffle}
  and \eqref{codegeneracy} for the shuffle product and for the action
  of the codegeneracies on $\Z\Delta^\bull_n$.
  \qed
\end{proof}

\section{The derived Maurer-Cartan locus}

In the previous section, we introduced the simplicial differential
graded cocommutative coalgebra $[n]\mapsto N_*(\Delta_n)$, and proved
that it was dual to the cosimplicial differential graded commutative
algebra $[n]\mapsto\Lambda^n$. As we have seen in \eqref{K}, the
inverse functor to the normalized cochains from cosimplicial abelian
groups to coconnective cochain complexes may be represented in terms
of $\Lambda^\bull$:
\begin{equation*}
  K^\bull(Z) = Z^0(Z\o\Lambda^\bull) .
\end{equation*}
Here, $Z^0(-)$ is the abelian group of $0$-cocycles in the tensor
product of $Z^*$ with the cosimplicial cochain complex
$\Lambda^\bull$.

An abelian differential graded Lie algebra is the same thing as a
cochain complex, and its Maurer-Cartan locus may be identified with
the space of $1$-cocycles of $L^*$. Thus, in this case, we obtain the
identification
\begin{equation*}
  \MC(L) \cong K^0(L_+[1]) ,
\end{equation*}
where we recall from the introduction that $L_+[1]$ is the suspended
cochain complex
\begin{equation*}
  L_+[1]^i =
  \begin{cases}
    L^{i+1} , & i\ge0 , \\
    0 , & i<0 .
  \end{cases}
\end{equation*}
Tensoring $L^*$ with $\Lambda^n$, we see that the functor $K^\bull$
may be identified with the derived Maurer-Cartan locus of
\emph{abelian} differential graded Lie algebras:
\begin{equation*}
  \MC^\bull(L) \cong K^\bull(L[1]) .
\end{equation*}
This provides some motivation for our definition of the derived
Maurer-Cartan locus for not necessarily abelian differential graded
Lie algebras.

The graded vector space $\Lambda^n$ decomposes as the direct sum of
the ideal $e_0\Lambda^n$ and the image of the coface map
$d^0\colon \Lambda^{n-1}\to\Lambda^n$. The monomials in the elements
$\eps_i=e_{i+1}-e_i$, $0\le i<n$, form a basis over $\Z$ of the
free abelian group $\Lambda^n$.
\begin{lemma}
  \label{natural}
  There is a natural isomorphism
  \begin{equation*}
    \MC^n(L) \cong \bigoplus_{k=0}^n \bigoplus_{0\le i_1<\dots<i_k<n}
    \eps_{i_1} \dots \eps_{i_k} \, L^{k+1} ,
  \end{equation*}
  induced by the projection
  $\Lambda^n\to\Lambda^n/e_0\Lambda^n\cong\im(d^0)$. The element of
  $\MC^n(L)$ corresponding to
  \begin{equation*}
    \xi = \sum_{k=0}^n \sum_{0\le i_1<\dots<i_k<n} \eps_{i_1} \dots
    \eps_{i_k} \, x_{i_1\dots i_k} , \quad  x_{i_1\dots i_k} \in L^{k+1} ,
  \end{equation*}
  equals $\xi-e_0F(\xi)$.
\end{lemma}
\begin{proof}
  An element of $L^*\o\Lambda^n$ of total degree $1$ has the form
  $\xi + e_0 \eta$, where
  \begin{equation*}
    \eta = \sum_{k=0}^n \sum_{0\le i_1<\dots<i_k<n} \eps_{i_1}\dots
    \eps_{i_k} \, y_{i_1\dots i_k} , \quad  y_{i_1\dots i_k} \in L^{k+2} .
  \end{equation*}
  Taking the curvature of this element, we obtain the expression
  \begin{equation*}
    F(\xi+e_0\eta) = \bigl( F(\xi) + \eta \bigr) - e_0 \bigl(
    \delta\eta + [\xi,\eta] \bigr) .
  \end{equation*}
  Along the vanishing locus of the equation $F(\xi)+\eta=0$, the
  equation $\delta\eta+[\xi,\eta]=0$ holds automatically: it is just
  the Bianchi identity \eqref{Bianchi} for the curvature.
  \qed
\end{proof}

In terms of this representation for $\MC^\bull(L)$, the codegeneracy
morphism $s^j\colon \MC^{n+1}(L)\to\MC^n(L)$ is given by the formula
\begin{equation}
  \label{grouplike}
  (s^jx){}_{i_1\dots i_k} =
  \begin{cases}
    x_{i_1\dots i_\ell i_{\ell+1}+1\dots i_k+1} + x_{i_1\dots
      i_{\ell-1}i_\ell+1\dots i_k+1} , & i_\ell+1=j , \\
    x_{i_1\dots i_\ell i_{\ell+1}+1\dots i_k+1} , & i_\ell+1<j\le i_{\ell+1} .
  \end{cases}
\end{equation}
When $j>0$, the coface morphism $d^j\colon \MC^{n-1}(L)\to\MC^n(L)$ is
given by the formula
\begin{equation}
  \label{face}
  (d^jx){}_{i_1\dots i_k} =
  \begin{cases}
    x_{i_1\dots i_\ell i_{\ell+1}-1\dots i_k-1} , & i_\ell+1<j\le
    i_{\ell+1} , \\
    0 , & j\in\{i_1+1,\dots,i_k+1\} .
  \end{cases}
\end{equation}
The remaining coface map $d^0$ encodes the geometry of the derived
Maurer-Cartan locus: it is given by the formula
\begin{multline}
  \label{d0}
  (d^0x){}_{i_1i_2\dots i_k} \\
  = \sum_{\substack{ \tau_1,\dots,\tau_k\in\{0,1\} \\ 
      i_j-\tau_j<i_{j+1}-\tau_{j+1} }} (-1)^{\tau_1+\dots+\tau_k+k}
  \left( x_{i_1-\tau_1\dots i_k-\tau_k} + \delta_{0i_1}
    F(\xi)_{i_2-\tau_2\dots i_k-\tau_k} \right) .
\end{multline}

In particular, the codegeneracy $s^0\colon\MC^1(L)\to\MC^0(L)$ is
given by the formula $s^0(x,y)=x$, and the face maps
\begin{equation*}
  d^0, d^1 \colon MC^0(L)\cong L^1 \too \MC^1(L)\cong L^1\times L^2
\end{equation*}
are given by the formulas $d^0x=(x,-F(x))$ and $d^1x=(x,0)$. Thus,
there is a natural identification of the classical locus
$\pi^0(\MC^\bull(L))$ of the cosimplicial scheme $\MC^\bull(L)$ with
the Maurer-Cartan locus $\MC(L)$ of the differential graded Lie
algebra $L^*$.

By \eqref{grouplike} and \eqref{face}, the codegeneracy maps $s^i$ of
$\MC^\bull(L)$ as well as the coface maps $d^i$, $i>0$, are
homomorphisms of (abelian) group schemes. Adapting the terminology of
Bousfield and Kan (\cite{BK}, Chapter~X, Section~4.8), we call such a
cosimplicial scheme \textbf{grouplike}. Grouplike cosimplicial spaces
are fibrant (op.\ cit.\ Section~4.6): we now show that an analogous
property holds for grouplike simplicial schemes.

Let $X^\bull$ be a cosimplicial scheme. The \textbf{matching scheme}
$M^n(X)$ is the equalizer
\begin{equation*}
  M^n(X) = \eq \Bigl(
  \begin{xy}
    \morphism|a|/{@{>}@<3pt>}/<1000,0>[\prod_{0\le i<n} X^{n-1}`%
    \prod_{0\le i<j<n} X^{n-2};]
    \morphism|b|/{@{>}@<-3pt>}/<1000,0>[\prod_{0\le i<n} X^{n-1}`%
    \prod_{0\le i<j<n} X^{n-2};]
  \end{xy} \Bigr)
\end{equation*}
where the two maps in this diagram take $(x^i)_{0\le i<n}$ to
$\bigl(s^ix^j\bigr)_{0\le i<j<n}$ and
$\bigl(s^{j-1}x^i\bigr)_{0\le i<j<n}$.

A cosimplicial scheme $X^\bull$ is \textbf{fibrant} if for each
$n\ge0$, the morphism $X^n\to M^n(X)$ given by the formula $x^i=s^ix$
is smooth. The proof of the following proposition is modeled on
Moore's proof that simplicial groups are fibrant.
\begin{proposition}[cf.\ \cite{BK}, Proposition 4.9]
  A grouplike cosimplicial scheme is fibrant.
\end{proposition}
\begin{proof}
  In characteristic zero, a morphism of group schemes is smooth if it
  has a section. We define morphisms $y^i\colon M^n(X)\to X^n$,
  $0\le i\le n+1$, by induction on $i$: $y^0=1$ and
  \begin{equation*}
    y^{i+1} = y^i d^i\bigl( (s^iy^i)^{-1} x^i \bigr) .
  \end{equation*}
  It is easily proved, by induction on $i$, that $s^jy^i=x^j$ for
  $j<i$. The desired section is $y^{n+1}\colon M^n(X)\to X^n$.
  \qed
\end{proof}

The graded commutative algebra $N_*(\CO(\MC^\bull(L)))^\sharp$ is
actually a free commutative algebra: there is a graded vector space
$V_*$ and an isomorphism of graded commutative algebras
\begin{equation*}
  N_*(\CO(\MC^\bull(L)))^\sharp \cong \Sym V .
\end{equation*}
This is a because the multiplicative structure of
$N_*(\CO(\MC^\bull(L)))$ does not depend on the coface maps of
$\MC^\bull(L)$, but only on its codegeneracies, and thus there is an
isomorphism
\begin{equation}
  \label{sharp}
  N_*(\CO(\MC^\bull(L)))^\sharp \cong
  N_*(\CO(\MC^\bull(L^\natural)))^\sharp ,
\end{equation}
where $L^\natural$ is the underlying cochain complex of the
differential graded Lie algebra $L^*$. We now apply the following result.
\begin{proposition}
  Let $W_*$ be a connective chain complex. Then there is a connective
  graded vector space $V_*$ and an isomorphism of graded commutative
  algebras
  \begin{equation*}
    N_*(\Sym K(W))^\sharp \cong \Sym V .
  \end{equation*}
\end{proposition}
\begin{proof}
  The proof makes use of the fact that $\Sym K(W)$ is a simplicial
  commutative bialgebra. Proposition~\ref{SW} implies that
  $N_*(\Sym K(W))^\sharp$ is a graded commutative bialgebra.

  Let $W_*^+$ be the chain complex
  \begin{equation*}
    W^+_i = \begin{cases}
      W_i , & i>0 , \\
      0 , & i=0 .
    \end{cases}
  \end{equation*}
  There is a natural isomorphism of graded commutative bialgebras
  \begin{equation*}
    N_*(\Sym K(W))^\sharp \cong \Sym W_0 \o N_*(\Sym K(W^+))^\sharp .
  \end{equation*}
  Let $\Nbar_*(\Sym K(W^+))$ be the augmentation ideal of
  $N_*(\Sym K(W^+))$, that is, the chain complex of elements of
  positive degree, and let
  \begin{equation*}
    Q(N_*(\Sym K(W^+))) \cong \Nbar_*(\Sym K(W^+)) \big/ \bigl(
    \Nbar_*(\Sym K(W^+))\cdot\Nbar_*(\Sym K(W^+) \bigr)
  \end{equation*}
  be the chain complex of indecomposables.

  Theorem~7.5 of Milnor and Moore~\cite{MM}, which holds over any
  field of characteristic zero, states that $N_*(\Sym K(W^+))^\sharp$
  is a free graded commutative algebra, generated by any section of
  the quotient morphism
  \begin{equation*}
    \Nbar_*(\Sym K(W^+))^\sharp \too Q(N_*(\Sym K(W^+)))^\sharp .
    \qed
  \end{equation*}
\end{proof}

\section{Proof of Theorem~\ref{main}}

\label{proof}

This section is the heart of this paper: we prove that the
differential graded Maurer-Cartan locus $\mc(L)$ is equivalent to the
derived Maurer-Cartan locus $\MC^\bull(L)$. Since both functors only
depend on $L^*_+$, we will assume in this section that $L^*=L^*_+$,
in other words, that $L^i$ vanishes unless $i\ge1$.

The normalization $N_*(\CO(\MC^\bull(L)))$ of the simplicial
commutative algebra $\CO(\MC^\bull(L))$ is the differential graded
commutative algebra of functions on an affine differential graded
scheme. If $n>0$, there is a natural linear map
$\alpha\mapsto\Phi(\alpha)$ from the vector space $(L^{n+1}){}^\vee$
to $\CO(\MC^n(L))$, which takes $\alpha\in(L^{n+1}){}^\vee$ to the
linear form
$\alpha(x_{0\dots n-1})$ on $\MC^n(L)$. (Here, we use the coordinate
system of Lemma~\ref{natural}.) The explicit formula \eqref{face} for
the coface maps $d^i\colon\MC^{n-1}(L))\to\MC^n(L)$, $1\le i\le n$,
shows that the function $\Phi(\alpha)$ lies in
\begin{equation*}
  \bigcap_{i=1}^n \, \ker\bigl( \p_i\colon \CO(\MC^n(L)) \to
  \CO(\MC^{n-1}(L)) \bigr) ,
\end{equation*}
and thus determines an element of $N_n(\CO(\MC^\bull(L)))$. The
resulting linear map from $\bigoplus_{n=0}^\infty(L^{n+1}){}^\vee$ to
$\bigoplus_{n=0}^\infty N_{-n}(\CO(\MC^\bull(L)))$ induces a morphism
of graded commutative algebras
\begin{equation}
  \label{Phi}
  \Phi \colon \CO(\mc(L)) \too N_{-*}(\CO(\MC^\bull(L))) .
\end{equation}

\begin{lemma}
  \label{compatible}
  The morphism $\Phi$ is compatible with the differentials on the
  differential graded algebras $\CO(\mc(L))$ and
  $N_{-*}(\CO(\MC^\bull(L)))$.
\end{lemma}
\begin{proof}
  The differential $d\colon \CO^{n-1}(\mc(L))\to\CO^n(\mc(L))$ is the
  sum of differentials $d_1$ and $d_2$, given by the formulas
  $(d_1\alpha)(x)=\alpha(\delta x)$ and
  $(d_2\alpha)(x,y)=(-1)^{|x|} \, \alpha([x,y])$, where $x,y\in L^*$.

  Using the explicit formula for the codegeneracy map
  \eqref{grouplike}, we may show that the product of the linear forms
  $\Phi(\beta)$ and $\Phi(\gamma)$ associated to the one-cochains
  $\beta\in(L^{p+1}){}^\vee$ and $\gamma\in(L^{q+1}){}^\vee$ on $L^*$
  is represented by the following quadratic polynomial on
  $\MC^{p+q}(L)$:
  \begin{equation*}
    \bigl( \Phi(\beta)\Phi(\gamma) \bigr)(x) =
    \sum_{\substack{ I=\{i_1<\dots<i_p\} \\
        J=\{j_1<\dots<j_q\} \\ I\coprod J=\{0,\dots,p+q-1\} }}
    (-1)^{\sum_{\ell=1}^p (i_\ell-\ell+1) } \beta(x_{i_1\dots i_p})
    \gamma(x_{j_1\dots j_q}) .
  \end{equation*}
  The differential
  \begin{equation*}
    d \colon N_{-n}(\CO(\MC^\bull(L))) \too
    N_{-n+1}(\CO(\MC^\bull(L)))
  \end{equation*}
  equals the pullback by the morphism
  $d^0\colon\MC^{n-1}(L)\to\MC^n(L)$. Applied to $\Phi(\alpha)$, where
  $\alpha\in\bigl(L^{n+1}\bigr){}^\vee$, \eqref{d0} gives
  \begin{equation*}
    d\Phi(\alpha)(x) = - \alpha\bigl(\delta x_{0\dots
      n-2}\bigr) - \tfrac{1}{2} \sum_{I\coprod J=\{0,\dots,n-2\}}
    (-1)^{\sum_{\ell=1}^{|I|} (i_\ell-\ell+1)} \, \alpha\bigl(
    [x_I,x_J] \bigr) .
  \end{equation*}
  The first and second terms inside the parentheses correspond to
  $d_1$ and $d_2$ respectively: in the case of $d_2$, we use the
  explicit formula for $\Phi(\beta)\Phi(\gamma)$ to make this
  identification.
  \qed
\end{proof}

Let $F^1\CO(\mc(L))$ be the augmentation ideal of the algebra of
Chevalley-Eilenberg cochains of $L^*$ (that is, cochains of negative
degree), and let $F^1\CO(\MC^\bull(L))$ be the augmentation ideal of
the simplicial commutative algebra $\CO(\MC^\bull(L))$ (that is,
polynomials with vanishing constant term). For $k>1$, let
$F^k\CO(\mc(L))$ and $F^k\CO(\MC^\bull(L))$ be the $k$th powers of
$F^1\CO(\mc(L))$, and $F^1\CO(\MC^\bull(L))$, and let
$F^kN_*(\CO(\MC^\bull(L))=N_*(F^k\CO(\MC^\bull(L)))$. The morphism
$\Phi$ of \eqref{Phi} is compatible with the filtrations on
$\CO(\mc(L))$ and $N_*(\CO(\MC^\bull(L))$, and the induced morphism
\begin{equation*}
  \gr_F\Phi \colon \gr_F\CO(\mc(L)) \too
  \gr_FN_{-*}(\CO(\MC^\bull(L)))
\end{equation*}
may be identified with the morphism
\begin{equation}
  \label{phinatural}
  \Phi \colon \CO(\mc(L^\natural)) \too
  N_{-*}(\CO(\MC^\bull(L^\natural)))
\end{equation}
of differential graded commutative algebras.  Theorem~\ref{main} is
thus a consequence of the following lemma.
\begin{lemma}
  The morphism \eqref{phinatural} is a quasi-isomorphism.
\end{lemma}
\begin{proof}
  Let $Z_*$ be the connective chain complex $L^\natural[1]{}^\vee$: as
  a vector space, we have $Z_n\cong\bigl(L^{n+1}\bigr){}^\vee$. The
  iterated shuffle product $g^{\o n}$ induces a morphism of simplicial
  abelian groups
  \begin{equation*}
    K_\bull g^{(n)} \colon K_\bull(Z^{\o n}) \too K_\bull(Z)^{\o n} ,
  \end{equation*}
  which is a quasi-isomorphism by the Eilenberg-Zilber theorem. This
  morphism is equivariant with respect to the action of the symmetric
  group $S_n$. Taking invariants, summing over $n$, and taking
  normalized chains, we obtain a quasi-isomorphism of differential
  graded commutative algebras
  \begin{equation*}
    \Sym Z \too N_*(\Sym K_\bull(Z)) .
  \end{equation*}
  This may be identified with the morphism $\Phi$ of
  \eqref{phinatural}.
  \qed
\end{proof}

\section{Generalization to nilpotent \Linf-algebras}

The definition of the differential graded scheme $\mc(L)$ extends to
\Linf-algebras: these are a generalization of differential graded Lie
algebras in which the Jacobi rule is only satisfied up to a hierarchy
of higher homotopies.

An operation $[x_1,\dots,x_k]$ on a graded vector space $L^*$ is
\textbf{graded antisymmetric} if
\begin{equation*}
  [x_1,\dotsc,x_i,x_{i+1},\dots,x_k] + (-1)^{|x_i||x_{i+1}|} \,
  [x_1,\dotsc,x_{i+1},x_i,\dots,x_k] = 0
\end{equation*}
for all $1\le i\le k-1$.

An \textbf{\Linf-algebra} is a graded vector space $L^*$ with graded
antisymmetric operations a sequence $[x_1,\dotsc,x_n]$ of degree
$2-n$, $n>0$, such that for each $n$,
\begin{equation*}
  \sum_{k=1}^n
  \sum_{\substack{ I=\{i_1<\dots<i_k\} \\ J=\{j_1<\dots<j_{n-k}\} \\
      I\cup J=\{1,\dots,n\} }} (-1)^{\eps+\sum_{i=1}^{n-k} (j_i-i)} \,
  [[x_{i_1},\dots,x_{i_k}],x_{j_1},\dotsc,x_{j_{n-k}}] = 0 .
\end{equation*}
Here, the sign $(-1)^\eps$ is the sign associated by the Koszul sign
convention to the action of $\pi$ on the elements $x_1,\dots,x_n$ of
$L^*$. The $1$-bracket $x\mapsto[x]$ is a differential on $L^*$, so an
\Linf-algebra is in particular a cochain complex.

An \Linf-algebra $L^*$ is \textbf{nilpotent} if it has a decreasing
filtration $F_kL^*$ such that for each $i\in\Z$, $F_kL^i=0$ if
$k\gg0$, and for each $n>0$,
\begin{equation*}
  [F_{k_1}L^{i_1},\dots,F_{k_n}L^{i_n}] \subset
  F_{k_1+\dots+k_n+1}L^{i_1+\dots+i_n-n+2} .
\end{equation*}
In particular, every differential graded Lie algebra concentrated in
degrees $\ge1$ is nilpotent.
\begin{lemma}
  Let $L^*$ be an \Linf-algebra. Its truncation $L^*_+$ is an
  \Linf-algebra, which is nilpotent if and only if the curvature
  \begin{equation*}
    F(x) = \sum_{n=1}^\infty \, \frac{1}{n!} \, [x^{\o n}]
  \end{equation*}
  is a polynomial map from $L^1$ to $L^2$.
\end{lemma}

The Maurer-Cartan locus $\MC(L)\subset L^1$ of an \Linf-algebra $L^*$
is the vanishing locus of the Maurer-Cartan equation $F(x)=0$. The
differential graded Maurer-Cartan locus $\mc(L)$ of $L$ is the affine
differential graded scheme with underlying scheme the affine space
$L^1$, and with the differential graded algebra of functions
\begin{equation*}
  \CO(\mc(L)) = \Sym (L_+[1]{}^\vee) .
\end{equation*}
The differential $d$ on $\CO(\mc(L))$ is the differential of the
generalization of the Chevalley-Eilenberg complex to \Linf-algebras:
it is the sum of the adjoints of the $n$-fold brackets $[-,\dots,-]$,
which maps $(L^k)^\vee$ to
\begin{equation*}
  \bigoplus_{\substack{i_1+\dots+i_n=k-n+2\\i_1,\dots,i_n\ge1}} \,
  (L^{i_1})^\vee \o \dots \o (L^{i_n})^\vee .
\end{equation*}

\begin{example}
  If $F\colon V\to W$ is an arbitrary polynomial which vanishes at
  $0\in V$, we may form an \Linf-algebra $L^*$ with $L^1=V$, $L^2=W$,
  and all other vector spaces $L^i$ vanishing. The brackets of $L^*$
  are the polarizations of the homogeneous components of the
  polynomial $F\colon L^1\to L^2$. The Maurer-Cartan locus of $L^*$ is
  the vanishing locus of the polynomial $F$, and $\CO(\mc(L))$ is the
  Koszul complex of $F$.
\end{example}

Thus, the differential graded Maurer-Cartan locus for nilpotent
\Linf-algebras generalizes at the same time the differential graded
Maurer-Cartan locus for differential graded Lie algebras and the
Koszul complex for a polynomial map between finite dimensional vector
spaces.

The tensor product $L^*\o\Lambda^n$ of a nilpotent \Linf-algebra $L^*$
with $\Lambda^n$ is again a nilpotent \Linf-algebra, with brackets
\begin{equation*}
  [ x_1\o\alpha_1 , \dots , x_n\o\alpha_n ] =
  (-1)^{\sum_{i>j}k_i\ell_j} \, [x_1,\dots,x_n] \o
  \alpha_1\dots\alpha_n
\end{equation*}
for $x_i\in L^{k_i}$ and $\alpha_i \in (\Lambda^n){}^{\ell_i}$.
\begin{definition}
  The derived Maurer-Cartan locus $\MC^\bull(L)$ of a nilpotent
  \Linf-algebra is the cosimplicial scheme
  \begin{equation*}
    \MC^n(L) = \MC(L\o\Lambda^n) .
  \end{equation*}
\end{definition}

The statement and proof of Theorem~\ref{main} extend without
difficulty to nilpotent \Linf-algebras. The only twist in the proof is
the verification of Lemma~\ref{compatible} in this more general
setting, that $\Phi$ is a morphism of complexes. We leave this task to
the motivated reader.

\end{document}